\documentclass{amsart}

\newtheorem{theorem}{Theorem}[section]
\newtheorem{lemma}[theorem]{Lemma}

\theoremstyle{definition}

\theoremstyle{remark}
\newtheorem{remark}[theorem]{Remark}

\theoremstyle{corollary}
\newtheorem{corollary}[theorem]{Corollary}

\numberwithin{equation}{section}

\begin{document}

\title[Gradient Bounds on compact Manifolds]{Gradient Bounds and Liouville theorems for Quasi-linear equations on compact Manifolds with nonnegative Ricci curvature}

\author{Dimitrios Gazoulis}
\address{Department of Mathematics, University of Athens (EKPA), Panepistemiopolis, 15784 Athens, Greece}
\email{dgazoulis@math.uoa.gr}

\author{George Zacharopoulos}
\address{Department of Mathematics, University of Athens (EKPA), Panepistemiopolis, 15784 Athens, Greece}
\email{gzacharop@math.uoa.gr}





\begin{abstract}
In this work we establish a gradient bound and Liouville-type theorems for solutions to Quasi-linear elliptic equations on compact Riemannian manifolds with nonnegative Ricci curvature. Also, we provide a local splitting theorem when the inequality in the gradient bound becomes equality at some point. Moreover, we prove a Harnack-type inequality and an ABP estimate for the gradient of solutions in domains contained in the manifold.
\end{abstract}

\maketitle



\section{Introduction}

Let $ ( \mathcal{M}, g) $ be a smooth Riemannian manifold. Throughout this paper we shall assume that $ (\mathcal{M}, g) $ or simply $ \mathcal{M} $ is a compact, connected, smooth and boundaryless Riemannian manifold of dimension $ n \geq 2 $ with nonnegative Ricci curvature. We will indicate otherwise if some of these assumptions are dropped. Also we denote by $\nabla$ the Levi-Civita connection with respect to the Riemannian metric $g$.

Consider the equation

\begin{equation}\label{GeneralQuasi-LinearEquation}
\begin{gathered}
div \left(  \Phi ' ( | \nabla u |^2) \nabla u \right) = F'(u)
\\ \textrm{where} \;\: u \: : \mathcal{M} \rightarrow \mathbb{R} \;\: \textrm{and} \;\: F \:, \: \Phi \: : \mathbb{R} \rightarrow \mathbb{R}
\end{gathered}
\end{equation}
This equation arise as the Euler-Lagrange of the energy functional
\begin{equation}\label{FunctionalForQuasi-LinearEq}
J(u) = \int_{\mathcal{M}} \left( \frac{1}{2} \Phi( |\nabla u|^2) + F(u) \right) d \mu_g
\end{equation}

In this paper we apply the ``$ P- $function technique'' to obtain a pointwise gradient estimate for the equation \eqref{GeneralQuasi-LinearEquation}. The method of proof is based on Maximum Principles and it has been introduced in \cite{Payne,PP,Sperb}. In addition, we establish Liouville-type theorems, one of which generalizes the constancy of bounded harmonic functions for Quasi-linear equations when a stability condition is assumed. Other applications are Harnack-type estimates and Alexandrov-Bekelman-Pucci type estimates for the gradient of solutions. In the last section, we also determine a local splitting result as an extension of \cite{FO} for Quasi-linear equations.

The idea of obtaining gradient bounds via the Maximum Principle turned out to be very effective and it found several applications in many topics including Riemannian geometry. To be more precise, relevant works can be found in \cite{CY,FO,FV,FV2,Hamilton,SZ} to cite a few. Furthermore, a novel approach to the Maximum Principle method has been recently exploited in a very successful way in \cite{Andrews,AC,AX}, in order to obtain oscillation and modulus of continuity estimates. Some of these estimates can provide gradient bounds, when the solutions are smooth enough. In particular, in \cite{AX}, they study a class of Quasi-Linear equations that include equations of the form \eqref{GeneralQuasi-LinearEquation} and obtain the gradient bound that appears in \cite{CGS} (see Theorem 1.6), for compact manifolds with non negative Ricci curvature (see Corollary 2). Theorem \ref{ThmGradientBoundForQuasi-LinearEq} in our work, is a special case of the bound in Corollary 2 in \cite{AX}, however our approach enable us to establish additional results, such as a weak Harnack inequality and an ABP estimate as mentioned previously.

Let now
\begin{equation}\label{PfunctionForeneralQuasi-LinearEquation}
P (u;x) = 2 \Phi'( | \nabla u(x) |^2) | \nabla u(x) |^2 - \Phi (| \nabla u(x) |^2) -2 F(u(x)) \;\;\;,\; x \in \mathcal{M}.
\end{equation}
The above quantity is the related $ P- $function of equation \eqref{GeneralQuasi-LinearEquation}. An abstract definition of $ P- $functions can be found in \cite{DG} and is similarly formulated for Riemannian manifolds. 

We assume that $ \Phi \in C^3( [0,+ \infty)) \;,\: F \geq 0 \;,\; F \in C^2 (\mathbb{R}) $ and $ \Phi(0) =0 $ and we define
\begin{equation}\label{Defa_ij=..}
a_{ij}(\sigma ) := 2 \Phi''(| \sigma |^2) \sigma_i \sigma_j + \Phi'(| \sigma |^2) \delta_{ij}
\end{equation}
and we suppose that one of the following conditions is satisfied: $ \\ $
\textbf{Assumption (A)} There exist $ p>1 \:,\; a \geq 0 $ and $ c_1,c_2 >0 $ such that for any $ \sigma , \: \xi \in \mathbb{R}^n \setminus \lbrace 0 \rbrace $,
\begin{equation}\label{AssumptionAona_ij1}
c_1(a + | \sigma |)^{p-2} \leq \Phi'( | \sigma |^2) \leq c_2 (a + | \sigma |)^{p-2} 
\end{equation} 
and
\begin{equation}\label{AssumptionAona_ij2}
c_1(a + | \sigma |)^{p-2} | \xi |^2 \leq \sum_{i,j=1}^n a_{ij}(\sigma ) \xi_i \xi_j \leq c_2 (a + | \sigma |)^{p-2} | \xi |^2
\end{equation}
$ \\ $
The structural assumptions \eqref{AssumptionAona_ij1} and \eqref{AssumptionAona_ij2} and considering $ p >1 $, allow us to apply regularity results from \cite{T}. So in assumption (A), we consider weak solutions of \eqref{GeneralQuasi-LinearEquation}, that is $ u \in W^{1,p}(\mathcal{M}) $ and $ p \in (1, + \infty) $.

$ \\ $
\textbf{Assumption (B)} There exist $ c_1, \: c_2 >0 $ such that for any $ \sigma \in \mathbb{R}^n $
\begin{equation}\label{AssumptionBona_ij1}
c_1(1 + | \sigma |)^{-1} \leq \Phi'( | \sigma |^2) \leq c_2 (1 + | \sigma |)^{-1} 
\end{equation} 
and
\begin{equation}\label{AssumptionBona_ij2}
c_1(1 + | \sigma |)^{-1} | \xi ' |^2 \leq \sum_{i,j=1}^n a_{ij}(\sigma ) \xi_i \xi_j \leq c_2 (1 + | \sigma |)^{-1} | \xi ' |^2
\end{equation}
for any $ \xi' = (\xi,\xi_{n+1}) \in \mathbb{R}^{n+1} $ which is orthogonal to $ ( - \sigma,1) \in \mathbb{R}^{n+1} . \\ \\ $
For assumption (B), we consider solutions $ u \in C^2(\mathcal{M}) . \\ $

The above assumptions (A) and (B) are classical, they agree for instance with the ones of \cite{CGS}, and examples of functional satisfying the above conditions are the Allen-Cahn equation, the $ p- $Laplacian (with $ p>1 $) and the mean curvature operators, which correspond to the cases
\begin{equation}\label{pLaplace}
\begin{gathered}
(i) \;\: \Phi(t) = t  \;\:,\; \Delta u = F'(u)  \\
\textrm{where} \;\: \Delta u = \frac{1}{\sqrt{det(g_{ij})}} \partial_k \left( \sqrt{det(g_{ij})} g^{kl} \partial_l u \right) ,
\end{gathered}
\end{equation}
\begin{equation}\label{AllenCahn}
\begin{gathered}
(ii) \;\: \Phi(t) = \frac{2}{p} t^{p/2} \;\:,\; \Delta_p u = F'(u)  \\
\textrm{where} \;\: \Delta_p u = \frac{1}{\sqrt{det(g_{ij})}} \partial_k \left( \sqrt{det(g_{ij})} g^{kl} \partial_l  (| \nabla u |^{p-2} \nabla u ) \right) ,
\end{gathered}
\end{equation}
\begin{equation}\label{MinimalSurface}
\begin{gathered}
(iii) \;\: \Phi(t) = 2 \sqrt{1+t} -2 \;\:,\; div ( \frac{\nabla u}{\sqrt{1+| \nabla u |^2}} ) = F'(u) \\
\textrm{where} \;\: div ( \frac{\nabla u}{\sqrt{1+| \nabla u |^2}} ) = \frac{1}{\sqrt{det(g_{ij})}} \partial_k \left( \sqrt{det(g_{ij})} g^{kl} \partial_l ( \frac{\nabla u}{\sqrt{1+| \nabla u |^2}} ) \right) ,
\end{gathered}
\end{equation}
written in local coordinates respectively.
$ \\ $

Note that, at first it seems that the specific choice $ \Phi(t) =  2 \sqrt{1+t} -2 $ fails to satisfy \eqref{AssumptionBona_ij2}. We have 
\begin{align*}
\frac{1+ | \sigma |}{| \xi '|^2} \sum_{i,j=1}^n a_{ij}(\sigma ) \xi_i \xi_j = \frac{1+ | \sigma |}{(1+| \sigma |^2)^{\frac{5}{2}}} \;\;\;,\;\; \textrm{where} \;\: \xi' = (\sigma , | \sigma |^2).
\end{align*}
by choosing $ \xi = \sigma $. This quantity is not bounded below by a positive constant as in \eqref{AssumptionBona_ij2} by sending $ | \sigma | \rightarrow + \infty $. However, in our case the boundedness of the gradient is automatically verified, since the solutions are smooth and the manifold $ \mathcal{M} $ is compact. So, the constants $c_1 $ and $ c_2 $ in \eqref{AssumptionBona_ij2} will depend on the bound of the gradient but this does not affect Theorem \ref{ThmGradientBoundForQuasi-LinearEq} since it is a sharp pointwise gradient estimate, it affects only the estimates in Theorems \ref{ThmHarnackForGradient} and \ref{ThmABPestimateforGradient}. In the statement of these theorems we refer to
\eqref{ProofHarnackEqNewConstants}, \eqref{ProofHarnackEqNewConstantsAs(B)} for the definition of the constants in the estimates and their dependence on the bound of the gradient. In \cite{CGS} they assume in assumption (B) that there exist a constant $ C=C( || u||_{L^{\infty}}) $ such that $ | \nabla u | \leq C $.

In the case of assumption (A), by Theorem 1 in \cite{T} we can avoid this dependence, since $ | \nabla u | $ will be bounded by a constant that depend on $ || u||_{L^{\infty}} $.

In this work we prove the following gradient bound for solutions of \eqref{GeneralQuasi-LinearEquation}. $ \\ $

\begin{theorem}\label{ThmGradientBoundForQuasi-LinearEq}
Let $ \mathcal{M} $ be a smooth and compact Riemannian manifold with nonnegative Ricci curvature. Let $ u $ be a solution of
\begin{equation}\label{ThmQuasi-LinearEquation}
div \left(  \Phi ' ( | \nabla u |^2) \nabla u \right) = F'(u)
\end{equation}
such that $ F \in C^2( \mathbb{R} ; [0, + \infty)) $ and assume that either assumption (A) or (B) holds.

Then
\begin{equation}\label{ThmGradientBoundStatementEq}
2 \Phi'( | \nabla u(x) |^2) | \nabla u(x) |^2 - \Phi (| \nabla u(x) |^2) \leq 2 F(u(x)) 
\end{equation}
for any $ x \in \mathcal{M} $.
\end{theorem}
$ \\ $
When $ \Phi' \equiv 1 $ in \eqref{GeneralQuasi-LinearEquation}, the respective $ P- $ function is $ P (u;x) = | \nabla u(x) |^2 - 2 F(u(x)) $ and the gradient bound becomes
\begin{align*} 
| \nabla u(x) |^2 \leq 2 F(u(x)) \;\;,\; \forall \: x \in \mathcal{M}
\end{align*}
This case is studied in \cite{FV}, where it is proved that $ P $ satisfies the following elliptic inequality
\begin{align*}
| \nabla u |^2 \Delta P - 2 F' \langle \nabla u , \nabla P \rangle \geq \frac{| \nabla P |^2 }{2} + 2 | \nabla u |^2 Ric( \nabla u, \nabla u)
\end{align*}

For the proof of the Theorem \ref{ThmGradientBoundForQuasi-LinearEq} we follow both \cite{CGS} and \cite{FV}. 

Our second main result is a Liouville-type theorem for solutions of \eqref{GeneralQuasi-LinearEquation} when $ F'' \geq 0 $. This assumption on $ F $ guarantees stability for any solution, to be more precise, the second variation of the energy functional $ J(u) = \int_{\mathcal{M}} \frac{1}{2} \Phi( | \nabla u |^2) + F(u) $ is non negative. $ \\ $

\begin{theorem}\label{ThmLiouvilleForQuasiLinear}
Let $ \mathcal{M} $ be a smooth and compact Riemannian manifold with nonnegative Ricci curvature. Let $ u \in C^3(\mathcal{M}) $ be a solution of
\begin{equation}\label{ThmLiouvQuasi-LinearEquation}
div \left(  \Phi ' ( | \nabla u |^2) \nabla u \right) = F'(u)
\end{equation}
such that $ F \in C^2 $ with $ F'' \geq 0 $ and assume that either assumption (A) or (B) holds.

Then $ u $ is a constant.
\end{theorem}

Note that if $ F $ convex and $ \Phi(t) = \frac{2}{p} t^{p/2} $, we obtain a Liouville-type result for the $ p- $Laplacian that generalizes the classical result that the only bounded harmonic functions on compact manifolds are the constant functions.

Another Liouville-type theorem is the following

\begin{theorem}\label{ThmLiouvilletype2}
Let $ u $ be a solution of \eqref{GeneralQuasi-LinearEquation} and suppose assumptions of Theorem \ref{ThmGradientBoundForQuasi-LinearEq} are satisfied. In addition, if assumption (A) holds and $ p \geq 2 $, we require that $ F(u) = O(| u-u_0 |^p ) $ as $ u \rightarrow u_0 $ for every $ u_0 \in \mathbb{R} $ such that $ F(u_0) =0 $. If there exists $ x_0 \in \mathcal{M} $ such that $ F(u(x_0)) =0 $, then $ u $ is a constant in $ \mathcal{M} $.
\end{theorem}

Also, we establish some gradient estimates for the solutions of \eqref{GeneralQuasi-LinearEquation}. We denote as $ B_R $ a ball of radius $ R $ and center any given point in $ \mathcal{M} $. For Theorems \ref{ThmHarnackForGradient} and \ref{ThmABPestimateforGradient} below, we require that the sectional curvature of $ \mathcal{M} $ is nonnegative in order to apply the results in \cite{Cabre}. First, a Harnack inequality for the gradient of solutions

\begin{theorem}\label{ThmHarnackForGradient}
Let $ \mathcal{M} $ be a smooth and compact Riemannian manifold with nonnegative sectional curvature and $ u \in C^3 ( \mathcal{M}) $ be a solution of
\begin{equation}\label{ThmLiouvQuasi-LinearEquation}
div \left(  \Phi ' ( | \nabla u |^2) \nabla u \right) = F'(u)
\end{equation}
such that $ F'' \leq 0 $. Assume that either: (i) assumption (A) holds and $ a>0 $ when $ p \neq 2 $ or (ii) assumption (B) holds.

Then
\begin{equation}\label{thmHarnackInequalityStatement}
\begin{gathered}
\frac{1}{| B_R |^{1/q}} \left( \int_{B_R} | \nabla u |^{2q} d \mu_g \right)^{1/q} \leq C ( \inf_{B_R} | \nabla u |^2 + \frac{R^2}{| B_{2R} |^{1/n}} || \: |Hes \: u|^2 \: ||_{L^{n}(B_{2R})}  \\  + \frac{R^2}{| B_{2R} |^{1/n}} || \Phi'(| \nabla u |^2 ) Ric(\nabla u, \nabla u) ) ||_{L^{n}(B_{2R})} )
\end{gathered}
\end{equation}
where $ q>0 $ and $ C $ depends on $ n \;, \tilde{c}_1 $ and $ \tilde{c}_2 $ defined either in \eqref{ProofHarnackEqNewConstants} or \eqref{ProofHarnackEqNewConstantsAs(B)} respectively.
\end{theorem}


Additionally, an Alexandrov-Bekelman-Pucci type estimate (ABP estimate) for the gradient of solutions is obtained. We  assume the following property for a given bounded domain (bounded, open and connected set) $ \Omega \subset \mathcal{M} $:
\begin{equation}\label{Gpropertyfordomains}
\begin{gathered}
\textrm{Given} \;\: R>0 \;\: \textrm{and} \;\: \theta \in (0,1), \;\: \textrm{it holds} \;\:
| B_R(x) \setminus \Omega | \geq \theta | B_R(x) | \;\;\;\; \forall x \in \Omega
\end{gathered}
\end{equation}
where $ B_R(x) $ is a ball of radius $ R $ and center $ x \in \mathcal{M} . \\ $

\begin{theorem}\label{ThmABPestimateforGradient}
Let $ \Omega \subset \mathcal{M} $ be a bounded domain and assume that \eqref{Gpropertyfordomains} holds for some constants $ R >0 $ and $ \theta \in (0,1) $ and $ \mathcal{M} $ has nonnegative sectional curvature. Let $ u \in C^3( \mathcal{M}) $ be a solution of \eqref{GeneralQuasi-LinearEquation} that satisfy $ \limsup_{x \rightarrow \partial \Omega} | \nabla u | = 0 $.

Then, for some $ z_0 \in \overline{\Omega} $,
\begin{equation}\label{ThmAbPestimateStatement}
\sup_{\Omega} | \nabla u |^2 \leq C_{\theta} \frac{R^2}{| B_{2R}(z_0) |^{1/n}} || F''(u) | \nabla u |^2 ||_{L^n(\Omega \cap B_{2R}(z_0))}
\end{equation}
where $ C_\theta $ is a constant depending on $ n \;, \tilde{c}_1 \;, 
\tilde{c}_2 $ and $ \theta $.
\end{theorem}

The constants $ \tilde{c}_1 \;, \tilde{c}_2 $ in Theorem \ref{ThmABPestimateforGradient} are defined in \eqref{ProofHarnackEqNewConstants} if assumption (A) holds or in \eqref{ProofHarnackEqNewConstantsAs(B)} if assumption (B) holds respectively.
$ \\ $
\begin{remark} Condition \eqref{Gpropertyfordomains} is satisfied by ``narrow'' domains, in the sense that if $ | \Omega| $ is small and set $ R:= \left( \dfrac{2| \Omega|}{\inf_{y\in \mathcal{M}} |B_1(y)|} \right)^{1/n} \leq 1 $, then $ \forall \: x \in \Omega $, it holds $ | B_R(x) \setminus \Omega | \geq 1/2 | B_R(x)| $ (for further details see Remark 8.4 in \cite{Cabre}).

Another example is the following: let $ B_R(x) $ is such that $ \Omega \subset B_R(x) \subset \mathcal{M} $, and 
$$ |\Omega| \leq \delta | B_R(x)| \:,\; \delta \in (0,1) \;,\: \textrm{then} \;\: | B_R(x) \setminus \Omega | \geq (1 - \delta) | B_R(x) |\;\: \textrm{and} \; \theta = 1 - \delta \in (0,1). $$
\end{remark}

Finally, in the last section we will prove that the existence of a nonconstant bounded solution $ u $ for which the gradient bound \eqref{ThmGradientBoundStatementEq} becomes equality at some point $ x_0 \in \mathcal{M} $, leads to a local splitting theorem as well as to a classification of such solution $ u $. This result is motivated by the work in \cite{FO}, in which they proved local and global splitting theorems for the Allen-Cahn equations when the equipartition of the energy holds at some point. 

\begin{remark}\label{RmkExtensionNonCompactCase}
In the case of complete noncompact manifolds one of the main technical difficulties to extend Theorem \ref{ThmGradientBoundForQuasi-LinearEq} is in general the lack of group addition property (the case of $ \mathbb{R}^n $ has been done in \cite{CGS}). Another technical difficulty is to ensure that $ \inf_{\mathcal{M}} | \nabla u| =0 $ holds. Also, in \cite{FV2} Theorem \ref{ThmGradientBoundForQuasi-LinearEq} has been proved for possibly unbounded domains in $ \mathbb{R}^n $ with nonnegative mean curvature. Theorems \ref{ThmLiouvilleForQuasiLinear}, \ref{ThmLiouvilletype2} and \ref{LocalSplittingTheorem} can be derived if Theorem \ref{ThmGradientBoundForQuasi-LinearEq} holds in the noncompact case. In that case, Theorem \ref{LocalSplittingTheorem} can be extended to a global splitting theorem by Cheeger-Gromoll splitting theorem for complete noncompact manifolds (see \cite{FO}).
\end{remark}

\section{Proof of Theorem \ref{ThmGradientBoundForQuasi-LinearEq}}

First we prove that $ P $ defined in \eqref{PfunctionForeneralQuasi-LinearEquation} is a $ P $ function of \eqref{GeneralQuasi-LinearEquation}, following the proof of Theorem 2.2 in \cite{CGS}.

\begin{lemma}\label{PfunctionEllipticEstimateForGrBound}
Consider $ \Omega \subset \mathcal{M} $ be a connected open set and let $ u $ be a solution of \eqref{GeneralQuasi-LinearEquation} such that $ \inf_{\overline{\Omega}}| \nabla u | >0 $. Assume that either assumption (A) or (B) holds.

Then
\begin{equation}\label{EllipticDifIneqforP}
| \nabla u |^2 \sum_{i,j} \nabla_j (d_{ij}(\nabla u)\nabla_i P) + \sum_{i}  B_i \nabla_i P \geq \frac{| \nabla P |^2}{2 \Lambda(| \nabla u|^2 )} + 2 | \nabla u |^2 \Phi'(| \nabla u |^2 )  Ric(\nabla u, \nabla u)
\end{equation}
for any $ x \in \Omega $, where
\begin{align*}
d_{ij}( \nabla u) = \dfrac{a_{ij}(\nabla u)}{\Lambda(| \nabla u |^2)} \;,\: \Lambda( | \nabla u |^2) = 2 \Phi''( | \nabla u |^2 ) | \nabla u |^2 + \Phi'(| \nabla u |^2) \;, \\ B_i = - \frac{2F'(u)}{\Lambda( | \nabla u |^2)} \nabla_i u -  \frac{2\Phi''(| \nabla u |^2 ) | \nabla u |^2 F'(u)}{\Lambda ( | \nabla u |^2) \Phi'(| \nabla u |^2)} \nabla_i u \;\;\;\;\;\;\;\;\;\;\;
\end{align*}
and $ P $ defined in \eqref{PfunctionForeneralQuasi-LinearEquation}.
\end{lemma}
$ \\ $
The left hand side of \eqref{EllipticDifIneqforP} is a uniformly elliptic operator since $ | \nabla u | $ is bounded and $ \inf_{\overline{\Omega}}| \nabla u | >0 $.
So, a direct consequence of Lemma \ref{PfunctionEllipticEstimateForGrBound} and the Strong Maximum Principle (see Theorem 8.19 in \cite{GT}), is the following.

\begin{corollary}\label{CorMaximumPrincipleforP} Suppose the assumptions of Lemma \ref{PfunctionEllipticEstimateForGrBound}. If there exists $ x_0 \in \Omega $ such that
\begin{equation}\label{MaximumPrincipleforPEq}
P(u; x_0) = \sup_{\Omega} P(u;x)
\end{equation}
then $ P $ is constant in $ \Omega $.
\end{corollary}
\begin{proof}[Proof of Lemma \ref{PfunctionEllipticEstimateForGrBound}]
To begin with, consider the case where assumption (A) holds. By Theorem 1 in \cite{T} (see also Theorem 3.1 in \cite{CGS}) we have that $ u \in C^{1,\alpha} $ and the compactness of $ \mathcal{M} $ in addition imply that $ \nabla u $ is bounded. Since $ \inf_{\overline{\Omega}}| \nabla u | >0 $, the equation \eqref{GeneralQuasi-LinearEquation} is uniformly elliptic and thus by Theorem 6.3, Chapter 4 in \cite{LU} we obtain that $ u \in C^{2, \alpha} ( \Omega) $.

The uniform ellipticity of equation \eqref{GeneralQuasi-LinearEquation} is similarly true when assumption (B) holds, utilizing also the fact that $ | \nabla u | $ is bounded (since $ \mathcal{M} $ is compact). So again, we conclude that $ u \in C^{2, \alpha} ( \Omega) $.

Throughout this proof we will use abstract index notation for the calculations. We consider the function  \eqref{PfunctionForeneralQuasi-LinearEquation} and take it's covariant derivative. 
\begin{equation}\label{GradofPfunction}
\begin{gathered}
\nabla_i P = \nabla_i (2 \Phi' (|\nabla u|^2)|\nabla u|^2 - \Phi (|\nabla u|^2) -2 F(u)  )
\\ =  4 \Phi'' ( | \nabla u |^2 ) g(\nabla_i \nabla u, \nabla u) |\nabla u|^2+ 4 \Phi' (|\nabla u|^2)g (\nabla_i \nabla u, \nabla u) \\ -2\Phi' (|\nabla u|^2) g(\nabla_i \nabla u, \nabla u) -2 F' (u) \nabla_i u \\
= 4 \Phi'' (|\nabla u|^2) |\nabla u|^2 g(\nabla_i \nabla u, \nabla u) + 2 \Phi' (|\nabla u|^2) g(\nabla_i \nabla u, \nabla u)-2 f(u)\nabla_i u
\end{gathered}
\end{equation}
where $F' =f$.
By denoting
\begin{equation}\label{LamdaQuantity}
\Lambda (|\nabla u|^2) = 2 \Phi'' (|\nabla u|^2)|\nabla u|^2 + \Phi' (|\nabla u|^2)
\end{equation}
we get
\begin{equation}\label{derivativeofP}
\nabla_i P = 2 \Lambda (|\nabla u|^2) g(\nabla_i \nabla u, \nabla u) -2 f(u) \nabla_i u
\end{equation}
Next we will multiply $\nabla_i P$ with $$d_{ij} (\nabla u) = \frac{a_{ij} (\nabla u)}{\Lambda (|\nabla u|^2)} $$ and take the covariant derivative of $d_{ij} (\nabla u)\nabla_i P$. 
\begin{align} \label{covariantderivativedP}
\begin{split}
\nabla_j (d_{ij}(\nabla u) \nabla_i P  ) = & \nabla_j (2 a_{ij} (\nabla u) g (\nabla_i \nabla u, \nabla u) - 2f(u) d_{ij} (\nabla u) \nabla_i u  ) \\
= & \nabla_j (2 a_{ij} (\nabla u)\nabla_i \nabla_k u  )\nabla_k u + 2 a_{ij} (\nabla u) \nabla_i \nabla_k u \nabla_j \nabla_k u \\
& - \nabla_j (2 d_{ij} (\nabla u) \nabla_i u )f(u) -2 f' (u)  d_{ij}(\nabla u) \nabla_ju \nabla_i u 
\end{split}
\end{align}

We note that the left hand side of \eqref{covariantderivativedP} should be a priori expressed in the sense of distributions, however, the $ C^{2, \alpha} $ regularity of the solution guarantees that $ \nabla_j (d_{ij}(\nabla u) \nabla_i P  ) $ is a continuous function (see equation \eqref{ContinuousRightHandSide}).

First we will compute $\nabla_j (d_{ij}  (\nabla u) \nabla_i u)$. Using \eqref{ThmQuasi-LinearEquation} we get
\begin{equation}\label{beforeclaim}
\begin{gathered}
\nabla_j (d_{ij}  (\nabla u) \nabla_i u) =  \nabla_j d_{ij} (\nabla u)\nabla_i u + d_{ij} (\nabla u) \nabla_j \nabla_i u \\
\ \ \ \ \ \ \ \ \ \ \ \ \ \ \ \ =  \nabla_j d_{ij}(\nabla u) \nabla_i u + \frac{f(u)}{\Lambda (|\nabla u|^2)}
\end{gathered}
\end{equation}
Claim: The following identity holds 
\begin{align*}
\nabla_j d_{ij} (\nabla u) \nabla_i u = \frac{2 \Phi''(|\nabla u|^2) }{\Lambda (|\nabla u|^2)} ( |\nabla u|^2 \Delta u - g(\nabla_i \nabla u, \nabla u)\nabla_i u )
\end{align*}
Proof of the claim:

We have 
\begin{align*}
\nabla_j d_{ij}(\nabla u) \nabla_i u = \frac{\nabla_j (a_{ij} (\nabla u)) \Lambda (|\nabla u|^2) \nabla_i u - a_{ij}(\nabla u) \nabla_j (\Lambda (|\nabla u|^2))\nabla_i u}{\Lambda (|\nabla u|^2)^2}  
\end{align*}
The numerator equals to 
\begin{equation*}
\begin{gathered}
\nabla_j (a_{ij} (\nabla u)) \Lambda (|\nabla u|^2) \nabla_i u -  a_{ij}(\nabla u) \nabla_j (\Lambda (|\nabla u|^2))\nabla_i u \\ = 8\Phi''' (|\nabla u|^2)g(\nabla_j \nabla u, \nabla u)|\nabla u|^4 \nabla_j u \Phi''(|\nabla u|^2) + 4 \Phi'' (|\nabla u|^2)g(\nabla_j \nabla u, \nabla u) |\nabla u|^2 \nabla_j u \Phi' (|\nabla u|^2) \\
+ 4 \Phi'' (|\nabla u|^2) \nabla_j \nabla_i u \nabla_j u \nabla_i u |\nabla u|^2 \Phi'' (|\nabla u|^2)
+2 \Phi'' (|\nabla u|^2)\nabla_j \nabla_i u \nabla_j u \nabla_i u \Phi' (|\nabla u|^2) \\ + 4 \Phi'' (|\nabla u|^2) \nabla_i u \nabla_j u \Delta u |\nabla u|^2 \Phi'' (|\nabla u|^2)
+ 2 \Phi'' (|\nabla u|^2)\nabla_i u \nabla_i u \Delta u \Phi' (|\nabla u|^2) \\
+4 \Phi'' (|\nabla u|^2) g(\nabla_j \nabla u, \nabla u) \nabla_j u |\nabla u|^2 \Phi'' (|\nabla u|^2)
 + 2 \Phi'' (|\nabla u|^2) g(\nabla_j \nabla u, \nabla u) \nabla_j u \Phi' (|\nabla u|^2) \\
 -12 g(\nabla_j \nabla u, \nabla u)\Phi'' (|\nabla u|^2)\nabla_i u \Phi'' (|\nabla u|^2) \nabla_i u \nabla_j u - 6 g(\nabla_i \nabla u, \nabla u) \Phi'' (|\nabla u|^2) \nabla_i u \Phi' (|\nabla u|^2) \\
 -8 g(\nabla_ j \nabla u, \nabla u) |\nabla u|^4 \Phi'''(|\nabla u|^2) \nabla_j u \Phi'' (|\nabla u|^2) 
 - 4 g(\nabla_i \nabla u, \nabla u) |\nabla u|^2 \Phi'' (|\nabla u|^2) \nabla_i u \Phi' (|\nabla u|^2)
\end{gathered}
\end{equation*}
after we cancel some terms we obtain.
\begin{equation}\label{ClaimProofEq}
\begin{gathered}
\nabla_j (a_{ij} (\nabla u)) \Lambda (|\nabla u|^2) \nabla_i u - a_{ij}(\nabla u) \nabla_j (\Lambda (|\nabla u|^2))\nabla_i u= 4 \Phi'' (|\nabla u|^2)^2\nabla_i u \nabla_i u \Delta u |\nabla u|^2 \\
+ 2 \Phi'' (|\nabla u|^2) \Phi' (|\nabla u|^2)\nabla_i \nabla_i u \Delta u  
 - 4\Phi''(|\nabla u|^2)^2  g(\nabla_j \nabla u, \nabla u) \nabla_i u \nabla_iu \nabla_j u \\
-2 \Phi'' (|\nabla u|^2) \Phi' (|\nabla u|^2) g(\nabla_i u \nabla u, \nabla u) \nabla_i u\\
 = 2 \Phi'' (|\nabla u|^2) \Delta u |\nabla u|^2(2 |\nabla u|^2 \Phi'' (|\nabla u|^2) + \Phi'(|\nabla u|^2) ) \\
 - 2\Phi'' (|\nabla u|^2)g(\nabla_ j \nabla u, \nabla u) \nabla_ju (2 |\nabla u|^2 \Phi'' (|\nabla u|^2) + \Phi'(|\nabla u|^2) )\\
 =   2 \Phi'' (|\nabla u|^2) \Delta u |\nabla u|^2 \Lambda (|\nabla u|^2) 
  -2 \Phi'' (|\nabla u|^2)g(\nabla_ j \nabla u, \nabla u) \nabla_ju  \Lambda (|\nabla u|^2) \\
  =  2 \Phi'' (|\nabla u|^2)\Lambda (|\nabla u|^2) (\Delta u |\nabla u|^2 - g(\nabla_j \nabla u , \nabla u)\nabla_ju )
\end{gathered}
\end{equation}
Therefore,
\begin{equation}\label{derivativeofd}
\nabla_j d_{ij}(\nabla u) \nabla_i u =  \frac{ 2 \Phi'' (|\nabla u|^2) (\Delta u |\nabla u|^2 - g(\nabla_j \nabla u , \nabla u)\nabla_ju )}{\Lambda (|\nabla u|^2)}
\end{equation}
This finishes the proof of the claim.

Using \eqref{derivativeofd} and \eqref{beforeclaim} we get, 
\begin{equation} \label{fulderivativeof d}
\nabla_j (d_{ij}  (\nabla u) \nabla_i u) =\frac{2 \Phi''(|\nabla u|^2) }{\Lambda (|\nabla u|^2)} ( |\nabla u|^2 \Delta u - g(\nabla_i \nabla u, \nabla u)\nabla_i u )+ \frac{f(u)}{\Lambda (|\nabla u|^2)}
\end{equation}
Next we will calculate the term $\nabla_j (2a_{ij} (\nabla u) \nabla_i \nabla_k u )\nabla_k u  $. We will prove the following 
\begin{equation} \label{derivativeofa}
\nabla_j (a_{ij}(\nabla u) \nabla_k \nabla_iu ) \nabla_k u=  \nabla_k (a_{ij} (\nabla u) \nabla_i \nabla_j u )\nabla_ku + \Phi' (|\nabla u|^2) R_{ki}\nabla_iu \nabla_k u 
\end{equation}
For the proof of \eqref{derivativeofa} we are going to compute $\nabla_j (a_{ij}(\nabla u) \nabla_k \nabla_iu ) $ and then use commuting covariant  derivative formula in order of the curvature tensor to appear.\\

Proof of \eqref{derivativeofa}:
\begin{equation}\label{proofofderivativeofaEq}
\begin{gathered}
\nabla_j (a_{ij}(\nabla u) \nabla_k \nabla_i u  )= \nabla_j (2 \Phi'' (|\nabla u|^2) \nabla_i u \nabla_j u \nabla_k \nabla_i u +\Phi' (|\nabla u|^2) \nabla_k \nabla_j u  )\\
= 4 \Phi''' (|\nabla u|^2) g(\nabla_j \nabla u, \nabla u) \nabla_i u \nabla_j u \nabla_k \nabla_i u  +2 \Phi'' (|\nabla u|^2) \nabla_j \nabla_i u \nabla_j u \nabla_k \nabla_i u \\
 +2 \Phi'' (|\nabla u|^2) \nabla_i u \nabla_j \nabla_j u \nabla_k \nabla_i u  +2 \Phi'' (|\nabla u|^2) \nabla_i u \nabla_j u \nabla_j \nabla_k \nabla_i u\\
 +2 \Phi'' (|\nabla u|^2) \nabla_i \nabla_j u \nabla_j u \nabla_k \nabla_i u  + \Phi' (|\nabla u|^2) \nabla_j \nabla_k \nabla_j u \\
= 4 \Phi''' (|\nabla u|^2) g(\nabla_j \nabla u, \nabla u) \nabla_i u \nabla_j u \nabla_k \nabla_i u   +2 \Phi'' (|\nabla u|^2) \nabla_j \nabla_i u \nabla_j u \nabla_k \nabla_i u 
\\ +  2 \Phi'' (|\nabla u|^2) \nabla_i \nabla_j u \nabla_j u \nabla_k \nabla_i u  +2 \Phi'' (|\nabla u|^2) \nabla_iu \nabla_ju \nabla_k \nabla_j \nabla_i u - 2 \Phi'' \nabla_i u\nabla_j u R_{jkip}\nabla_p u \\ + 2 \Phi'' (|\nabla u|^2) \nabla_i u \nabla_j \nabla_j u \nabla_k \nabla_i u + \Phi' (|\nabla u|^2)\nabla_k \nabla_j \nabla_j u + \Phi' (|\nabla u|^2)R_{kp}\nabla_p u \\
(\textrm{and interchanging the indices} \;\: i \;\: \textrm{and} \;\: j \;\: \textrm{in the second and third term})
\\
=\nabla_k(2 \Phi''( | \nabla u |^2 ) \nabla_j u \nabla_i u \nabla_j \nabla_i u + \Phi'( | \nabla u |^2) \nabla_j \nabla_j u) + \Phi' (|\nabla u|^2)R_{kp}\nabla_p u
\\ = \nabla_k (a_{ij} (\nabla u) \nabla_j \nabla_i u )+ \Phi' (|\nabla u|^2)R_{kp}\nabla_p u
\end{gathered}
\end{equation}
where $R_{ijkl}$ are the components of the curvature tensor while $R_{ij} = R_{ippj}$ are the components of the Ricci tensor. 
Note that from the skew-symmetry  of the Riemann tensor  the term $- 2 \Phi'' \nabla_i u\nabla_j u R_{jkip}\nabla_p u= 0$ and so we showed \eqref{derivativeofa}. Also by taking the covariant derivative of $a_{ij} (\nabla u)\nabla_i \nabla_j u= f(u)$
along the direction $k$ we get $ \nabla_k (a_{ij} (\nabla u)\nabla_i \nabla_j u) = f'(u) \nabla_ku$. As a result \eqref{derivativeofa} becomes
\begin{equation} \label{fulderivativeofa}
\nabla_j (a_{ij}(\nabla u) \nabla_k \nabla_iu ) \nabla_k u= f'(u) |\nabla  u|^2 + \Phi' (|\nabla u|^2) R_{ki}\nabla_iu \nabla_k u 
\end{equation}
Combining \eqref{covariantderivativedP}, \eqref{fulderivativeof d} and \eqref{fulderivativeofa} we obtain
\begin{equation}\label{ContinuousRightHandSide}
\begin{gathered}
\nabla_j (d_{ij} (\nabla u) \nabla_i P )= 2 a_{ij}(\nabla u)\nabla_i \nabla_k u \nabla_j \nabla_k u + 2 \Phi' (|\nabla u|^2)R_{ij}\nabla_iu \nabla_ju \\
- \frac{4 \Phi''(| \nabla u |^2) f(u)}{\Lambda(| \nabla u |^2)} ( | \nabla u |^2 \Delta u - g( \nabla_i \nabla u, \nabla u) \nabla_i u ) -2 \frac{f^2 (u)}{\Lambda (|\nabla u|^2)}
\end{gathered}
\end{equation}
Here, the right hand side is a continuous function and so is the term $ \nabla_j (d_{ij} (\nabla u) \nabla_i P ) $. To verify this, recall that $ u \in C^{2, \alpha} ( \Omega) \;,\: \Phi \in C^3 ( \mathbb{R}) \;,\: F \in C^2 ( \mathbb{R}) $ and $ R_{ij} $ is smooth.

It follows directly that from \eqref{GeneralQuasi-LinearEquation} and \eqref{Defa_ij=..},
\begin{equation}\label{Laplacianterm}
\Delta u = \frac{f(u)}{\Phi' (|\nabla u|^2)} - \frac{ 2 \Phi'' (|\nabla u|^2)}{\Phi' (|\nabla u|^2)}\nabla_i u \nabla_j u \nabla_i \nabla_j u  
\end{equation}
so by \eqref{derivativeofP} and \eqref{Laplacianterm},
\begin{equation}\label{TermInvolvingNablaP}
 | \nabla u |^2 \Delta u - g( \nabla_i \nabla u, \nabla u) \nabla_i u  = -\frac{1}{2 \Phi' (|\nabla u|^2)} \nabla_i u \nabla_i P
\end{equation}
and thus by \eqref{TermInvolvingNablaP} and \eqref{ContinuousRightHandSide} we get 
\begin{align*}
\nabla_j (d_{ij}(\nabla u) \nabla_i P ) - \frac{2 \Phi'' (|\nabla u|^2) f(u)}{\Lambda (|\nabla u|^2) \Phi' (|\nabla u|^2)} \nabla_i u \nabla_i P= 2 a_{ij}(\nabla u)\nabla_i \nabla_k u\nabla_j \nabla_k u \\ + 2 \Phi' (|\nabla u|^2)R_{ij}\nabla_i u\nabla_j u-2 \frac{f^2(u)}{\Lambda (|\nabla u|^2)} \;\;\;\;\;\;\;\;\;\;\;\;\;\;\;\;\;\;\;\;\;\;\;\;\;\;
\end{align*}
From the Cauchy-Schwarz inequality we have
\begin{align*}
|\nabla u|^2 \nabla_i \nabla_k u \nabla_i \nabla_k u \geq \nabla_i \nabla_k u\nabla_i u\nabla_j \nabla_ku\nabla_j u
\end{align*}
and we apply this to the term $a_{ij}(\nabla u) \nabla_i \nabla_k u \nabla_j \nabla_ku$ so we get
\begin{align*}
a_{ij}(\nabla u) \nabla_i \nabla_k u \nabla_j \nabla_ku \geq  \frac{\Lambda(|\nabla u|^2)}{|\nabla u|^2} \nabla_i \nabla_k u\nabla_i u \nabla_j \nabla_k u\nabla_j u
\end{align*}
Also  directly from \eqref{derivativeofP} we obtain
$$\nabla_i \nabla_k u\nabla_i u\nabla_j \nabla_k u\nabla_j u = \frac{(\nabla_kP +2 f(u)\nabla_k u)(\nabla_kP +2 f(u)\nabla_ku )}{4 | \nabla u |^2 \Lambda (|\nabla u|^2)}$$
This concludes the proof of Lemma 2.1.

\end{proof}

\begin{proof}[Proof of Theorem \ref{ThmGradientBoundForQuasi-LinearEq}]
$ \\ $
We now complete the proof of Theorem \ref{ThmGradientBoundForQuasi-LinearEq}. We argue as in \cite{CGS} and \cite{FV} with the appropriate modifications.

Let $ u $ be a solution of \eqref{ThmQuasi-LinearEquation}, by Theorem 1 in \cite{T} (see also Theorem 3.1 in \cite{CGS}) we have that $ u \in C^{1, \alpha} ( \mathcal{M}) $.

Consider the set
\begin{equation}\label{ProofTheorem1.1Eq1}
\mathcal{E} : = \lbrace v \in C^{1, \alpha}( \mathcal{M}) \;\: \textrm{solution of} \;\: \eqref{ThmQuasi-LinearEquation} \;\: \textrm{such that} \;\: || v ||_{C^{1,\alpha}(\mathcal{M})} \leq || u ||_{C^{1,\alpha}(\mathcal{M})} \rbrace 
\end{equation}
Let $ P = P(u;x) $ defined in \eqref{PfunctionForeneralQuasi-LinearEquation} and consider
\begin{equation}\label{ProofTheorem1.1Eq2}
P_0 = \sup \lbrace P(v;x) \: | \; v \in \mathcal{E} \;,\; x \in \mathcal{M} \rbrace
\end{equation}
For proving the bound \eqref{ThmGradientBoundStatementEq} it suffices to prove that
\begin{equation}\label{ProofTheorem1.1Eq3}
P_0 \leq 0
\end{equation}
We argue by contradiction and suppose that
\begin{equation}\label{ProofTheorem1.1Eq4}
P_0 > 0
\end{equation}
So, there exist two sequences $ v_k \in \mathcal{E} $ and $ x_k \in \mathcal{M} $ such that
\begin{equation}\label{ProofTheorem1.1Eq5}
P_0 - \frac{1}{k} \leq P(v_k ; x_k) \leq P_0
\end{equation}
By the compactness of $ \mathcal{M} $, $ x_k $ converges to some $ x_0 \in \mathcal{M} $ up to a subsequence that we still denote as $ x_k $. 

In addition, by the uniform bound
\begin{equation}\label{ProofTheorem1.1Eq6}
|| v_k ||_{C^{1,\alpha}(\mathcal{M})} \leq || u ||_{C^{1,\alpha}(\mathcal{M})}
\end{equation}
and the Ascoli-Arzela theorem for compact manifolds (see for instance \cite{Petersen}), we have that $ v_k $ converges uniformly in $ C^1( \mathcal{M}) $ to some $ v_0 \in \mathcal{E} $, up to a subsequence.

Thus,
\begin{equation}\label{ProofTheorem1.1Eq7}
P_0 = \lim_{k \rightarrow + \infty} P(v_k;x_k) = P(v_0;x_0)
\end{equation}
by \eqref{ProofTheorem1.1Eq5}. 

This gives
\begin{equation}\label{ProofTheorem1.1Eq8}
0< P_0 = 2 \Phi'( | \nabla v_0(x_0) |^2) | \nabla v_0(x_0) |^2 - \Phi (| \nabla v_0(x_0) |^2) -2 F(v_0(x_0))
\end{equation}
and since $ F \geq 0 $ and $ \Phi(0) =0 $, it holds that
\begin{equation}\label{ProofTheorem1.1Eq9}
\nabla v_0 (x_0) \neq 0
\end{equation}
Now, consider the closed set
\begin{equation}\label{ProofTheorem1.1Eq9'}
\mathcal{N} = \lbrace x \in \mathcal{M} \: : \: P(v_0 ;x) = P_0 \rbrace
\end{equation}
The set $ \mathcal{N} $ is non empty since $ x_0 \in \mathcal{N} $ and is also open. Indeed, let $ y_1 \in \mathcal{N} $,
from \eqref{ProofTheorem1.1Eq9}, it holds that $ \inf_{\overline{B}_\delta (y_1)} | \nabla v_0 | >0 $ for some $ \delta >0 $, where $ B_\delta (y_1) $ is a ball with center $ y_1 $ and radius $ \delta $. So, utilizing Corollary \ref{CorMaximumPrincipleforP}, we obtain
\begin{equation}\label{ProofTheorem1.1Eq10}
P(v_0;x) = P_0 \;\;\;\; \textrm{for all} \;\: x \in B_\delta (y_1)
\end{equation}
which imply $ B_\delta (y_1) \subset \mathcal{N} $. By connectedness it follows
\begin{equation}\label{ProofTheorem1.1Eq10'}
P(v_0;x) = P_0 \;\;\;\; \textrm{for all} \;\: x \in \mathcal{M}.
\end{equation}

On the other hand, since $ \mathcal{M} $ is compact and $ v_0 \in C^1(\mathcal{M}) $ there exists $ y_0 \in \mathcal{M} $ in which $ v_0 $ attains it's minimum, and thus
\begin{equation}\label{ProofTheorem1.1Eq11}
\nabla v_0(y_0) = 0
\end{equation}
but then
\begin{equation}\label{ProofTheorem1.1Eq12}
P_0 = P(v_0;y_0) = -2 F(v_0(y_0)) \leq 0
\end{equation}
and contradicts \eqref{ProofTheorem1.1Eq4}. 

Therefore $ P_0 \leq 0 $ and we conclude.
\end{proof}


\section{Proof of Theorems \ref{ThmLiouvilleForQuasiLinear} and \ref{ThmLiouvilletype2}}

In this section we will prove the Liouville-type theorems. We begin with an appropriate elliptic inequality. However, in this case, this inequality is satisfied by the quantity $ P = | \nabla u|^2 . $

\begin{lemma}\label{PfunctionEllipticEstimateForLiouville}
Let $ u \in C^3(\mathcal{M}) $ be a solution of \eqref{GeneralQuasi-LinearEquation}.

Then either
\begin{equation} \label{ellipticinequalityforP=grad2u}
\begin{gathered}
\nabla_j (a_{ij} \nabla_i P ) \geq 2 F'' (u) |\nabla u|^2 + 2 c_1 (a+ | \nabla u |)^{p-2} |Hes \: u|^2 \\ + 2 \Phi' ( |\nabla u|^2) R_{ik} \nabla_i u \nabla_k u
\end{gathered}
\end{equation}
if assumption (A) holds,

or
\begin{equation}
\label{ellipticinequalityforP=grad2u2}
\nabla_j (a_{ij} \nabla_i P ) \geq 2 F'' (u) |\nabla u|^2 + 2 c_1 \frac{ |Hes \: u|^2}{1 + | \nabla u |} + 2 \Phi' ( |\nabla u|^2) R_{ik} \nabla_i u \nabla_k u
\end{equation}
if assumption (B) holds, where $ P = | \nabla u|^2 $.
\end{lemma}
$ \\ $
\textbf{Note:} Note that for proving Theorem \ref{ThmLiouvilleForQuasiLinear} we utilize Lemma \ref{PfunctionEllipticEstimateForLiouville} in the points where $ | \nabla u | $ is bounded away from zero (see Corollary \ref{CorMaximumPrincipleforP}).


\begin{proof}
We have the equation $div( \Phi' (|\nabla u|^2) \nabla u) = F' (u)$ which can be written as $a_{ij} (\nabla u) \nabla_i \nabla_j u = F' (u)$. Here

$$ a_{ij} (\sigma) = 2 \Phi'' (|\sigma|^2) \nabla_i \sigma \nabla_j \sigma + \Phi' (|\sigma|^2) \delta_{ij} $$ 

Denote by $P = |\nabla u|^2$ then
\begin{align*}
\nabla_i P = & \nabla_i (g(\nabla u, \nabla u)) =  2 g(\nabla_i \nabla u, \nabla u) 
=  2\nabla_i \nabla_k u\nabla_k u
\end{align*}
and
\begin{equation}\label{ProofofLiouvilleSecondDerivativeofP}
\begin{gathered}
\nabla_j \nabla_i P = \nabla_j (2 \nabla_i u \nabla_k u \nabla_k u ) \\
 = 2 \nabla_j \nabla_i \nabla_k u \nabla_k u + 2 \nabla_i \nabla_k u \nabla_j \nabla_k u 
\end{gathered}
\end{equation}

We claim that 
\begin{align*}
\nabla_j (a_{ij}(\nabla u) \nabla_k \nabla_i u  )- \nabla_k (a_{ij} (\nabla u)\nabla_i \nabla_j u )  =  \Phi' (|\nabla u|^2)R_{ik}\nabla_k u
\end{align*}
Proof of the claim 
\begin{equation}\label{PfClaim1}
\begin{gathered}
\nabla_j (a_{ij} (\nabla u)\nabla_k \nabla_i u )=  \nabla_j (2 \Phi'' (|\nabla u|^2) \nabla_i u \nabla_j u\nabla_k \nabla_i u + \Phi' (|\nabla u|^2) \nabla_k  \nabla_j u ) \\
=  4 \Phi''' (|\nabla u|^2) \nabla_j \nabla_l u \nabla_l u \nabla_i u \nabla_j u \nabla_k \nabla_i u + 2 \Phi'' ( |\nabla u|^2) \nabla_j \nabla_i u\nabla_j u \nabla_k \nabla_i u \\
+ 2 \Phi'' (|\nabla u|^2) \nabla_i u \nabla_j \nabla_j u \nabla_k \nabla_i u + 2 \Phi'' (|\nabla u|^2) \nabla_i u \nabla_j u \nabla_j \nabla_k \nabla_i u \\
 + 2 \Phi'' (|\nabla u|^2) \nabla_j \nabla_l u \nabla_lu \nabla_k \nabla_j u + \Phi' (|\nabla u|^2) \nabla_j \nabla_k \nabla_j u \\
= 4 \Phi''' (|\nabla u|^2) \nabla_j \nabla_l u \nabla_l u \nabla_i u \nabla_j u \nabla_k \nabla_i u + 2 \Phi'' ( |\nabla u|^2) \nabla_j \nabla_i u\nabla_j u \nabla_k \nabla_i u\\
+ 2 \Phi'' (|\nabla u|^2) \nabla_i u \nabla_j \nabla_j u \nabla_k \nabla_i u + 2 \Phi'' (|\nabla u|^2) \nabla_i u \nabla_j u \nabla_k \nabla_j \nabla_i u + 2 \Phi'' (|\nabla u|^2) \nabla_iu \nabla_j u R_{jkil} \nabla_l u \\
 + 2 \Phi'' (|\nabla u|^2) \nabla_j \nabla_l u \nabla_lu \nabla_k \nabla_j u + \Phi' (|\nabla u|^2) \nabla_k \nabla_j \nabla_j u + \Phi' (|\nabla u|^2) R_{jk} \nabla_j u \\
=  \nabla_k (a_{ij} (\nabla u)\nabla_i \nabla_j u ) + \Phi' (|\nabla u|^2) R_{ik} \nabla_i u
\end{gathered}
\end{equation}
where again $R_{ijkl}$ are the components of the Riemann curvature tensor and $R_{ij}$ are the components of the Ricci tensor.

We will now show that \eqref{ellipticinequalityforP=grad2u} holds. We utilize that $$a_{ij} \nabla_j \nabla_i P = 2a_{ij}(\nabla u)\nabla_j \nabla_i \nabla_k u \nabla_k u + 2a_{ij} (\nabla u)\nabla_i \nabla_k u \nabla_j \nabla_k u $$
So,
\begin{align*}
\nabla_j (a_{ij} \nabla_i P )= & 2 \nabla_j (a_{ij} (\nabla u) \nabla_i \nabla_k u ) \nabla_k u +2 a_{ij} (\nabla u)\nabla_i \nabla_k u \nabla_j \nabla_k u\\
=& 2 \nabla_k (a_{ij} (\nabla u)\nabla_i \nabla_j u ) \nabla_k u + 2 \Phi' ( |\nabla u|^2) R_{ik} \nabla_i u \nabla_k u+2 a_{ij} (\nabla u)\nabla_i \nabla_k u \nabla_j \nabla_k u\\
=& 2 F'' (u) \nabla_k u \nabla_k u + 2 a_{ij}\nabla_i \nabla_k u \nabla_j \nabla_k u + 2 \Phi' ( |\nabla u|^2) R_{ik} \nabla_i u \nabla_k u \\
= & 2 F'' (u) |\nabla u|^2 + 2 a_{ij}\nabla_i \nabla_k u \nabla_j \nabla_k u + 2 \Phi' ( |\nabla u|^2) R_{ik} \nabla_i u \nabla_k u \\
\end{align*}

At this point, we first consider the case where assumption (A) is satisfied. In this case, we have 
\begin{equation}\label{ProofLemEllipticIneqAss(A)}
\begin{gathered}
\nabla_j (a_{ij} \nabla_i P ) \geq 2 F'' (u) |\nabla u|^2 + 2 c_1 (a+ | \nabla u |)^{p-2} |Hes \: u|^2 \\ + 2 \Phi' ( |\nabla u|^2) R_{ik} \nabla_i u \nabla_k u
\end{gathered}
\end{equation}

Now, if assumption (B) holds, set $ \xi'= ( \xi, \xi_{n+1}) $ where $ \xi_i = \nabla_i \nabla_k u $ and $ \xi_{n+1} = \xi \cdot \nabla u $, so we 
%
%
similarly obtain
\begin{equation}
\label{ProofLemEllipticIneqAss(B)}
\begin{gathered}
\nabla_j (a_{ij} \nabla_i P ) \geq 2 F'' (u) |\nabla u|^2 + 2 c_1 \frac{ |Hes \: u|^2}{1 + | \nabla u |} \\ + 2 \Phi' ( |\nabla u|^2) R_{ik} \nabla_i u \nabla_k u
\end{gathered}
\end{equation}
\end{proof}

\begin{proof}[Proof of Theorem \ref{ThmLiouvilleForQuasiLinear}]
To complete the proof of Theorem \ref{ThmLiouvilleForQuasiLinear}, we argue as in the proof of Theorem \ref{ThmGradientBoundForQuasi-LinearEq} in the previous section. Observe that Corollary \ref{CorMaximumPrincipleforP} can also be applied in this case for $ P = | \nabla u |^2 $, utilizing Lemma \ref{PfunctionEllipticEstimateForLiouville} (in a ball such that $ | \nabla u | $ is bounded away from zero) and the Strong Maximum Principle (see Theorem 8.19 in \cite{GT}). Therefore $ P \leq 0 $ and we conclude.
\end{proof}

Now we proceed to the proof of Theorem \ref{ThmLiouvilletype2}. This result is the analog of Theorem 1.8 in \cite{CGS} for compact manifolds (see also Theorem 21 in \cite{AX}) and is independent of Lemma \ref{PfunctionEllipticEstimateForLiouville}. 

\begin{proof}[Proof of Theorem \ref{ThmLiouvilletype2}]
For the proof of Theorem \ref{ThmLiouvilletype2} we argue as in the proof of Theorem 1.8 in \cite{CGS}, utilizing the gradient bound obtained in the Theorem \ref{ThmGradientBoundForQuasi-LinearEq}. The only difference is that we define $$ \phi (t) = u(x_1 + t \: exp(v)) -u (x_0) $$ where $ exp : T_{x_1} \mathcal{M} \rightarrow \mathcal{M} $ is the exponential map and similarly conclude that $ \phi $ is identically zero for every $ x_1 \in \mathcal{M} $.
\end{proof}



\section{Proof of Theorems \ref{ThmHarnackForGradient} and \ref{ThmABPestimateforGradient}}

In this section, arguing similarly to section 5 in \cite{DG}, we will prove the gradient estimates, that is a Harnack-type inequality and an ABP estimate for the gradient of solutions.
For proving Theorem \ref{ThmHarnackForGradient}, in the case where the ellipticity condition in Assumption (A) is satisfied, we further assume that $ a>0 $ when $ p \neq 2 $. We could drop this assumption by considering a balls $ B_R $ such that $ \nabla u $ does not vanish. 
$ \\ $

\begin{proof}[Proof of Theorem \ref{ThmHarnackForGradient}]
We begin as follows. First we obtain an elliptic inequality for the quantity $ P= | \nabla u|^2 $ similar to that of Lemma \ref{PfunctionEllipticEstimateForLiouville} for assumption (A). In particular, arguing as in the proof of Lemma \ref{PfunctionEllipticEstimateForLiouville} and since $ F'' \leq 0 $, we have
\begin{equation}\label{ProofHarnackEq1}
\sum_{i,j} \nabla_j (a_{ij}(\nabla u) \nabla_i P)\leq 2 c_2 (a+ | \nabla u |)^{p-2} |Hes \: u|^2 + 2 \Phi'(| \nabla u |^2 ) Ric(\nabla u, \nabla u)
\end{equation}
where $ P= | \nabla u|^2 $.

At this point, the compactness of $ \mathcal{M} $ gives $ | \nabla u | \leq M $, and by \eqref{AssumptionAona_ij2} we observe that
\begin{equation}\label{ProofHarnackNewEllipticityfora_ijAs(A)}
\begin{gathered}
\tilde{c}_1 | \xi |^2 \leq a_{ij} \xi_i \xi_j \leq \tilde{c}_2 | \xi |^2 \\ \textrm{where} \;\: \tilde{c}_1 \;, \: \tilde{c}_2 \;\: \textrm{depend on} \:\; c_1 \;, c_2 \;,  a \;\: \textrm{and} \;\: M,
\end{gathered}
\end{equation}
since $ a>0 $ for $ p \neq 2 $ (and $ \sigma = \nabla u $). In particular, $ \tilde{c}_1 \;, \tilde{c}_2 $ can be written as
\begin{equation}\label{ProofHarnackEqNewConstants}
\begin{gathered} \tilde{c}_1 = c_1 a^{p-2} \;\; \textrm{and} \;\: \tilde{c}_2 = c_2 ( a + M)^{p-2} \;\;,\; \textrm{if} \;\: p >2 \\ \tilde{c}_1 = c_1 ( a + M)^{p-2} \;\; \textrm{and} \;\: \tilde{c}_2 = c_2 a^{p-2} \;\;,\; \textrm{if} \;\: p <2 \\
\tilde{c}_1 = c_1 \;\; \textrm{and} \;\: \tilde{c}_2 =c_2 \;\;,\: \textrm{if} \;\: p =2.
\end{gathered}
\end{equation}

So, utilizing the Harnack inequality on manifolds for $ P = | \nabla u|^2 $, in particular Theorem 8.1 in \cite{Cabre}, we conclude.
The compactness of $ \mathcal{M} $ and the smoothness of $ u $ guarantee the boundedness of the Hessian of $ | \nabla u|^2 $. $ \\ $

In the case where assumption (B) holds, we have the following ellipticity condition for $ a_{ij} $
\begin{equation}\label{ProofHarnackNewEllipticityfora_ijAs(B)}
\begin{gathered}
c_1 | \xi |^2  \leq a_{ij} \xi_i \xi_j \leq c_2 | \xi' |^2 \leq c_2 (1 + M^2 ) | \xi |^2 \\ \textrm{where} \;\:  \xi'=( \xi, \xi_{n+1}) \;\: \textrm{and} \;\: \xi_{n+1} = \xi \cdot \nabla u
\end{gathered}
\end{equation}
since $ \xi' \perp ( - \nabla u,1) $ and $ | \nabla u | \leq M $. In this case the ellipticity constants $ \tilde{c}_1 $ and $ \tilde{c}_2 $ will depend on $ \mathcal{M} $ and are defined as
\begin{equation}\label{ProofHarnackEqNewConstantsAs(B)}
\tilde{c}_1 = c_1 \;\: \textrm{and}  \:\; \tilde{c}_2 = c_2 (1 + M^2 )
\end{equation}

Therefore by \eqref{ProofHarnackEq1}
\begin{equation}\label{ProofHarnackEq4}
\begin{gathered}
\nabla_j (a_{ij} \nabla_i P ) \leq \tilde{c}_2  |Hes \: u|^2 + 2 \Phi' ( |\nabla u|^2) R_{ik} \nabla_i u \nabla_k u
\end{gathered}
\end{equation}
where $ \tilde{c}_2 = 2 c_2 (1 + M) $ and we apply Theorem 8.1 in \cite{Cabre} to conclude.
\end{proof}

\begin{proof}[Proof of Theorem \ref{ThmABPestimateforGradient}]
The proof of Theorem \ref{ThmABPestimateforGradient} is established with similar arguments as in the proof of Theorem \ref{ThmHarnackForGradient}. Particularly, by utilizing Lemma \ref{PfunctionEllipticEstimateForLiouville}, in both cases where either assumption (A) or (B) is satisfied, we have
\begin{equation}\label{ProofThm1.5Eq1}
\nabla_j (a_{ij} \nabla_i P)  \geq 2 F''(u) | \nabla u |^2
\end{equation}
since the nonnegativity of the sectional curvature implies that the Ricci curvature is nonnegative.
So, by Theorem 2.3 in \cite{Cabre} and \eqref{ProofThm1.5Eq1} we conclude.
\end{proof}


\section{A Local Splitting Theorem}

In this last section we will prove a local splitting theorem of the manifold in a neighborhood of $ x_0 $ together with a precise description of the solution in this neighborhood. The intuition behind these types of results when dealing with $ \mathbb{R}^n $, arises from the fact that if the equipartition of the energy of the Allen-Cahn functional $ \int \frac{1}{2}| \nabla u |^2 + F(u) $ holds true at a single point (i.e. $ \frac{1}{2} | \nabla u |^2 = F(u) $), then the solutions are one dimensional. This has been generalized in \cite{FO} for complete manifolds where local and global splitting theorems are proved when the equipartition of the energy holds at some point. We will prove the local splitting analog for Quasi-linear equations and we point out the main obstruction for extending the global splitting theorem in our case in Remark \ref{RmkNoncompactCase}.

The local splitting theorem is the following

\begin{theorem}\label{LocalSplittingTheorem}
Let $ u \in C^3 ( \mathcal{M}) $ and assume that equality is achieved in \eqref{ThmGradientBoundStatementEq} at a regular point $ x_0 $, i.e. $ \nabla u(x_0) \neq 0 $. Then, $ \\ $
(i) equality in \eqref{ThmGradientBoundStatementEq} holds in the connected component of $ \mathcal{M} \cap \lbrace \nabla u \neq 0 \rbrace $ that contains $ x_0 . \\ $
(ii) $ Ric( \nabla u, \nabla u ) $ vanishes at the connected component of $ \mathcal{M} \cap \lbrace \nabla u \neq 0 \rbrace $ that contains $ x_0 . \\ $
(iii) there is a neighborhood of $ U_{x_0} \subset \mathcal{M} $ of $ x_0 $, that splits as the Riemannian product $ \mathcal{N} \times I $ where $ \mathcal{N} \subset \mathcal{M} $ is a totally geodesic and isoparametric hypersurface with $ Ric( \mathcal{N}) \geq 0 $ and $ I \subset \mathbb{R} $ is an interval, $ \\ $
(iv) the solution $ u $ restricted to the neighborhood $ U_{x_0} $, is equal to $ u(p,s) = \phi(s) $ where $ \phi $ is a bounded and strictly monotone solution of $ \phi'' = \dfrac{F'( \phi)}{\Lambda( (\phi')^2)} $ and $ \Lambda (t) = 2 t \Phi''(t) + \Phi'(t) $.
\end{theorem}

For the proof of Theorem \ref{LocalSplittingTheorem} we utilize the techniques in \cite{FO} with some modifications. For the convenience of the reader we provide the details. 
\begin{proof}
Assume that $ u $ is a non constant solution of \eqref{GeneralQuasi-LinearEquation} and set
\begin{align*}
P(u;x) = 2 \Phi'( | \nabla u(x) |^2) | \nabla u(x) |^2 - \Phi (| \nabla u(x) |^2) -2 F(u(x)) \;\;\;,\; x \in \mathcal{M}.
\end{align*}
By Lemma \ref{PfunctionEllipticEstimateForGrBound} the following inequality holds
\begin{align}\label{ProofLocalSplittingEq1}
|\nabla u|^2 \nabla_j (d_{ij}(\nabla u) \nabla_i P  ) + B_i \nabla_i P \geq \frac{|\nabla P |^2}{2 \Lambda (|\nabla u|^2)}+2 | \nabla u |^2  \Phi' (|\nabla u |^2)Ric(\nabla u, \nabla u)
\end{align}
Theorem \ref{ThmGradientBoundForQuasi-LinearEq} gives that $ P \leq 0 $ on $ \mathcal{M} $
and therefore by \eqref{ProofLocalSplittingEq1} and the strong maximum principle it holds that
\begin{align*}
P(u;x) = P(u;x_0) =0 \;\;\;\;\;\;\;\;\;\;\;\;\;\;\;\;\;\;\;\;\;\;\;\;\;\;\;\;\;\;\;\;\; \\ \textrm{in the connected component of} \;\: \mathcal{M} \cap \lbrace \nabla u \neq 0 \rbrace \;\: \textrm{that contains} \:\; x_0.
\end{align*}
since $ | \nabla u(x_0) | >0 $ by assumption.

Also, from \eqref{ProofLocalSplittingEq1} we have that $ Ric (\nabla u, \nabla u) =0 $ in the connected component of $ \mathcal{M} \cap \lbrace \nabla_g u \neq 0 \rbrace \;\: $ that contains $ x_0$.

We now proceed to the remaining statements of Theorem \ref{LocalSplittingTheorem}. Define $$ w = Q (u) = \int_{u_0}^{u} G(s)^{- \frac{1}{2}}ds $$ where $ G(s) = \Psi^{-1}(2F(s)) $. 

A direct calculation gives $|\nabla w| =1 $ and $\Delta w =0$ (see also Theorem 5.1 in \cite{CGS}). That is the function $w$ is harmonic and has constant gradient of length one, so it generates a parallel vector field. From local splitting theorems (see \cite{Petersen} or section 3 in \cite{FO}) we conclude.

The one dimensionality and monotonicity of the solutions restricted to this neighborhood is straightforward.

\end{proof}

\begin{remark}\label{RmkNoncompactCase}
Note that the extension of the above result to a global splitting theorem for compact manifolds would require an analogous result to that of Cheeger-Gromoll splitting theorem for complete noncompact manifolds. However such result is not known in general and so, it will possibly be a motivation for future research.
\end{remark}

\textbf{Acknowledgements}: We would like to thank Professor A. Farina for sharing his work in personal communication with D.G. and thus motivated a part of this paper.
Also, we would like to thank the anonymous referee for the valuable suggestions that improved the presentation and the content of the paper.
The first author acknowledges the ``Basic research Financing'' under the National Recovery and Resilience Plan ``Greece 2.0'' funded by the European Union-NextGeneration EU (H.F.R.I. Project Number: 016097).


\end{document}